\documentclass[a4paper,12pt]{amsart}
\usepackage{amsmath,amsthm,amssymb,latexsym,epic}
\usepackage{graphicx,enumerate}
\usepackage{comment}

\newtheorem{theorem}{Theorem}
\newtheorem{lemma}[theorem]{Lemma}
\newtheorem{remark}[theorem]{Remark}
\newtheorem{corollary}[theorem]{Corollary}
\newtheorem{proposition}[theorem]{Proposition}
\newtheorem{example}[theorem]{Example}

\usepackage[all]{xy}

\begin{document}
\title{Symmetric quasi-hereditary envelopes}
\author{Volodymyr Mazorchuk and Vanessa Miemietz}
\date{\today}

\maketitle

\begin{abstract}
We show how any finite-dimensional algebra can be realized as
an idempotent subquotient of  some symmetric  quasi-hereditary algebra.
In the special case of rigid symmetric algebras we show that they 
can be realized as centralizer  subalgebras of symmetric 
quasi-hereditary algebras. We also show that the infinite-dimensional
symmetric quasi-hereditary algebras we construct admit quasi-hereditary
structure with respect to two opposite orders, that they have 
strong exact Borel and $\Delta$-subalgebras and the corresponding 
triangular decompositions.
\end{abstract}

\section{Introduction}\label{s1}

The classical result of Dlab and Ringel (see \cite{DR}) says that 
every finite-dimensional algebra can be realized as a centralizer 
subalgebra of  some quasi-hereditary algebra. Motivated by the discovery 
of  (infinite-dimensional) symmetric quasi-hereditary algebras in
\cite{Pe} (see also \cite{CT,MT1,MT2,BS2}) we address the question 
whether every symmetric finite-dimensional algebra can be realized 
as a centralizer  subalgebra of some symmetric quasi-hereditary algebra
(we will loosely call the latter algebra a symmetric quasi-hereditary envelope
of the original algebra, although it is not uniquely defined in any 
reasonable sense). Unless the original algebra is semisimple,
any symmetric quasi-hereditary envelope must be infinite-dimensional. 

In the present paper we generalize the construction from \cite{DR}
and produce quasi-hereditary envelopes of finite-dimensional algebras.
Under some mild natural restrictions on the original algebra, 
quasi-hereditary envelopes of symmetric algebras turn out to be symmetric. 
In particular, we show that every symmetric and rigid algebra can be 
realized as a centralizer  subalgebra of some symmetric quasi-hereditary 
algebra. Furthermore, we show that every finite-dimensional algebra can
be realized as an idempotent subquotient of  some symmetric 
quasi-hereditary algebra. In particular, this gives many new
examples of symmetric quasi-hereditary algebras.

The infinite-dimensional (symmetric) quasi-hereditary algebras, which
we construct, have many interesting properties. To start with, all these
algebras are quasi-hereditary with respect to two natural orders
(one of them being the opposite of the other one). The standard and
costandard modules for these structures have a natural description
in terms of the original algebra. We also show that all these algebras 
have $\Delta$-subalgebras in the sense of K{\"o}nig (\cite{Ko,Ko2}).  
Assuming that the original algebra is graded, we show that our algebras
have a strong exact Borel subalgebra in the sense of K{\"o}nig
(\cite{Ko,Ko2}), as well as the corresponding triangular decompositions.

The paper is organized as follows: In Section~\ref{s2} we extend 
the construction from \cite{DR} to produce quasi-hereditary envelopes
of finite-dimensional algebra and show that these envelops are 
quasi-hereditary with respect to two natural opposite orders.
Finite-dimensional algebras are realized as centralizer subalgebras 
of their quasi-hereditary envelopes. In Section~\ref{s3} we prove that for 
symmetric rigid  finite-dimensional algebras the quasi-hereditary envelopes,
constructed in Section~\ref{s2}, are symmetric as well. For arbitrary algebras
we show how the construction can be generalized to realize 
every finite-dimensional algebra as an idempotent subquotient of  
some symmetric  quasi-hereditary algebra. In Section~\ref{s4} we describe
strong exact Borel and $\Delta$-subalgebras and the corresponding 
triangular decompositions for our infinite-dimensional
quasi-hereditary algebras. Finally, in Section~\ref{s5} we discuss some examples, in particular, those coming from Schur algebras
and the BGG category $\mathcal{O}$.

\vspace{0.5cm}

\noindent
{\bf Acknowledgments.} The first author was partially supported
by the Swedish Research Council. We thank Serge Ovsienko for many 
stimulating and helpful discussions.

\section{Preliminaries}\label{s2}

Let $\mathbb{N}$ denote the set of positive integers, $\Bbbk$ be an
algebraically closed  field  and $\mathcal{A}$ be a basic 
$\Bbbk$-linear category with at most countably many objects and 
finite-dimensional projectives and injectives (see \cite{MOS} for details).
We will often loosely call such categories ``algebras'' (as they can be
realized using infinite-dimensional associative quiver algebras which do not
have a unit element in the general case) and use for them standard 
matrix notation with infinite matrices. 
For $x\in \mathcal{A}$ we denote by $\mathtt{e}_x$ 
the identity element in $\mathcal{A}(x,x)$.
 
Assume that for some $N\in \mathbb{N}$ we have a (fixed) finite filtration of 
$\mathcal A$ by two-sided ideals as follows:
\begin{equation}\label{eq1}
\mathcal A=\mathcal I_0 \supsetneq \mathcal I_1\supsetneq \mathcal 
I_2\supsetneq \dots \supsetneq\mathcal I_N=0. 
\end{equation}
Assume further that $\mathcal I_i\mathcal I_j\subset \mathcal I_{i+j}$
and that $\mathcal I_i/\mathcal I_{i+1}$ 
are semi-simple as $\mathcal A$-bimodules.

Consider the new category $\mathfrak{A}$, whose objects are
$x[i]$, $x\in \mathcal{A}$, $i\in\mathbb{Z}$. For $x,y\in \mathcal{A}$
and $i,j\in \mathbb{Z}$ set $\mathfrak{A}(x[i],y[j])=\mathcal{A}(x,y)$.
Then the multiplication  in $\mathcal{A}$ induces a multiplication in
$\mathfrak{A}$, which makes $\mathfrak{A}$ into a category. The category
$\mathfrak{A}$ comes together with the natural action of $\mathbb{Z}$ 
by autoequivalences via shifts $[i]$, $i\in\mathbb{Z}$ (here $[1]$ means 
``shift by one to the right'').  The category $\mathfrak{A}$ can be seen 
as a $\mathbb{Z}$-Morita-equivalent extension of $\mathcal{A}$ 
(every object in $\mathcal{A}$ is repeated $|\mathbb{Z}|$ times). 
We shall think of $\mathfrak{A}$ also as of infinite matrices of the form
\begin{displaymath}
\left( 
\begin{array}{cccccc}
\ddots & \vdots & \vdots & \vdots & \vdots & \ddots\\
\dots & \mathcal{A} & \mathcal{A} & \mathcal{A} & \mathcal{A} & \dots\\
\dots & \mathcal{A} & \mathcal{A} & \mathcal{A} & \mathcal{A} & \dots\\
\dots & \mathcal{A} & \mathcal{A} & \mathcal{A} & \mathcal{A} & \dots\\
\dots & \mathcal{A} & \mathcal{A} & \mathcal{A} & \mathcal{A} & \dots\\
\ddots & \vdots & \vdots & \vdots & \vdots & \ddots\\
\end{array}
\right).
\end{displaymath}

Denote by $\mathfrak{B}$ the subcategory of $\mathfrak{A}$, which contains
all objects but only the following morphisms: For $x,y\in \mathcal{A}$
and $i,j\in \mathbb{Z}$ set
\begin{displaymath}
\mathfrak{B}(x[i],y[j])= 
\begin{cases}
\mathfrak{A}(x[i],y[j]), & i\geq j;\\
\mathcal I_{j-i}(x,y), & \text{otherwise}.
\end{cases}
\end{displaymath}
One can think of $\mathfrak{B}$ also as of
infinite matrices of the form
\begin{displaymath}
\left( 
\begin{array}{cccccc}
\ddots & \vdots & \vdots & \vdots & \vdots & \ddots\\
\dots & \mathcal{A} & \mathcal{A} & \mathcal{A} & \mathcal{A} & \dots\\
\dots & \mathcal{I}_1 & 
\mathcal{A} & \mathcal{A} & \mathcal{A} & \dots\\
\dots & \mathcal{I}_2 & \mathcal{I}_1 & 
\mathcal{A} & \mathcal{A} & \dots\\
\dots & \mathcal{I}_3 & \mathcal{I}_2 & 
\mathcal{I}_1 & \mathcal{A} & \dots\\
\ddots & \vdots & \vdots & \vdots & \vdots & \ddots\\
\end{array}
\right).
\end{displaymath}
Consider the subset  $\mathfrak{I}$ of $\mathfrak{B}$ with the same
set of objects and morphisms given by
\begin{displaymath}
\mathfrak{I}(x[i],y[j])= 
\begin{cases}
\mathcal I_{N-(i-j)}(x,y), & i-N<j<i;\\
\mathfrak{B}(x,y), & j\leq i-N;\\
0, & \text{otherwise}.
\end{cases}
\end{displaymath}
The set $\mathfrak{I}$ is not a subcategory as it does not contain
identity morphisms on objects. One can think of $\mathfrak{I}$ also as of
infinite matrices of the form
\begin{displaymath}
\left( 
\begin{array}{cccccc}
\ddots & \vdots & \vdots & \vdots & \vdots & \ddots\\
\dots & 0 & \mathcal{I}_{N-1} & \mathcal{I}_{N-2} 
& \mathcal{I}_{N-3} & \dots\\
\dots & 0 &  0 & \mathcal{I}_{N-1} & 
\mathcal{I}_{N-2} & \dots\\
\dots & 0 & 0 & 0 & \mathcal{I}_{N-1} & \dots\\
\dots & 0 & 0 &  0 & 0 & \dots\\
\ddots & \vdots & \vdots & \vdots & \vdots & \ddots\\
\end{array}
\right)
\end{displaymath}
It is easy to see that $\mathfrak{I}$ is an ideal of 
$\mathfrak{B}$. 
Define the category $\mathfrak{C}=\mathfrak{C}(\mathcal{A})=
\mathfrak{B}/\mathfrak{I}$. One can think of $\mathfrak{C}$ as of
infinite matrices of the form
\begin{displaymath}
\left( 
\begin{array}{cccccc}
\ddots & \vdots & \vdots & \vdots & \vdots & \ddots\\
\dots & \mathcal{A} & \mathcal{A}/\mathcal{I}_{N-1} & 
\mathcal{A}/\mathcal{I}_{N-2} & 
\mathcal{A}/\mathcal{I}_{N-3} & \dots\\
\dots & \mathcal{I}_1 & 
\mathcal{A} & \mathcal{A}/\mathcal{I}_{N-1} & 
\mathcal{A}/\mathcal{I}_{N-2} & \dots\\
\dots & \mathcal{I}_2 & \mathcal{I}_1 & 
\mathcal{A} & \mathcal{A}/\mathcal{I}_{N-1} & \dots\\
\dots & \mathcal{I}_3 & \mathcal{I}_2 & 
\mathcal{I}_1 & \mathcal{A} & \dots\\
\ddots & \vdots & \vdots & \vdots & \vdots & \ddots\\
\end{array}
\right).
\end{displaymath}
Observe that, given $x\in \mathcal{I}_i$ and some
class $a+\mathcal{I}_j\in \mathcal{A}/\mathcal{I}_{j}$ we have 
$x(a+\mathcal{I}_j)\subset xa+\mathcal{I}_{i+j}$ due to our assumption that $\mathcal{I}_i\mathcal{I}_j \subseteq\mathcal{I}_{i+j}$, so multiplication of these matrices is well-defined. Note that, using the matrix notation, left
modules are columns, while right modules are rows.

We consider two natural linear orders on $\mathbb{Z}$, we call the order 
where $i<i+1$ the first order, and the one where $i>i+1$ the second order.
These orders induce partial orders on the equivalence classes of primitive
idempotents in $\mathfrak{C}(\mathcal{A})$, which we will also call the
first and the second orders, respectively.
The following statement is a generalization of the main construction from
\cite{DR}:

\begin{proposition}\label{prop1}
\begin{enumerate}[(i)]
\item\label{prop1-2}
Left standard modules in the first order are given by direct 
summands of the following modules:
\begin{displaymath}
\Delta_{\mathfrak{C}}^{1,l}=
\left( 
\begin{array}{c}
\vdots \\
\mathcal{A}/\mathcal{I}_1 \\
\mathcal{A}/\mathcal{I}_1\\
\mathcal{A}/\mathcal{I}_1 \\
 0 \\
 0 \\
\vdots
\end{array}
\right)
\end{displaymath}
\item\label{prop1-3}
Left standard modules in the second order are given by direct summands of
the following module:
\begin{displaymath}
\Delta_{\mathfrak{C}}^{2,l}=\left( 
\begin{array}{c}
\vdots \\
 0 \\
 0 \\
\mathcal{A}/\mathcal{I}_1 \\
\mathcal{I}_1/\mathcal{I}_2\\
\mathcal{I}_2/\mathcal{I}_3\\
\vdots
\end{array}\right).
\end{displaymath}
\item\label{prop1-4}
Right standard modules for the first order are given by direct 
summands of the following module:
\begin{displaymath}
\Delta_\mathfrak{C}^{1,r}=
\left( 
\begin{array}{ccccccc}
\dots & \mathcal{I}_2/\mathcal{I}_3 & 
\mathcal{I}_1/\mathcal{I}_2 & 
\mathcal{A}/\mathcal{I}_1 & 0 & 0 & \dots\\
\end{array}
\right)
\end{displaymath}
\item\label{prop1-5}
Right standard modules for the second order are given by direct 
summands of the following module:
\begin{displaymath}
\Delta_\mathfrak{C}^{2,r}=\left( 
\begin{array}{ccccccc}
\dots & 0 & 0 &\mathcal{A}/\mathcal{I}_1  &\mathcal{A}/\mathcal{I}_1 & \mathcal{A}/\mathcal{I}_1& \dots\\
\end{array}\right)
\end{displaymath}
\item\label{prop1-1}
The category $\mathfrak{C}$ is quasi-hereditary with respect to both 
orders.
\end{enumerate}
\end{proposition}

\begin{proof}
Let $i\in\mathbb{Z}$. For the first order, the quotient of 
$\mathfrak{C}$ modulo the two-sided ideal, generated by all idempotents 
$\mathtt{e}_x[j]$, $x\in\mathcal{A}$, $j\in\mathbb{Z}$, $j>i$, 
looks as follows:
\begin{displaymath}
\left( 
\begin{array}{cccccc}
\ddots & \vdots & \vdots & \vdots & \vdots & \ddots\\
\dots & * & * & \mathcal{A}/\mathcal{I}_1 & 0 & \dots\\
\dots & * & * & \mathcal{A}/\mathcal{I}_1 & 0 & \dots\\
\dots & \mathcal{I}_2/\mathcal{I}_3 & 
\mathcal{I}_1/\mathcal{I}_2 & 
\mathcal{A}/\mathcal{I}_1 & 0 & \dots\\
\dots & 0 & 0 & 0 & 0 & \dots\\
\ddots & \vdots & \vdots & \vdots & \vdots & \ddots\\
\end{array}
\right)
\end{displaymath}
(here we do not care about the asterisks).

Similarly, for the second order, the quotient of $\mathfrak{C}$ modulo 
the two-sided ideal, generated by all idempotents $\mathtt{e}_x[j]$,
$x\in\mathcal{A}$, $j\in\mathbb{Z}$,  $j<i$, looks as follows:
\begin{displaymath}
\left( 
\begin{array}{cccccc}
\ddots & \vdots & \vdots & \vdots & \vdots & \ddots\\
\dots & 0 & 0 & 0 & 0 & \dots\\
\dots & 0 &\mathcal{A}/\mathcal{I}_1  & \mathcal{A}/\mathcal{I}_1 &  \mathcal{A}/\mathcal{I}_1& \dots\\
\dots & 0 &  \mathcal{I}_1/\mathcal{I}_2& * & * & \dots\\
\dots & 0& \mathcal{I}_2/\mathcal{I}_3 & 
* & * & \dots\\
\ddots & \vdots & \vdots & \vdots & \vdots & \ddots\\
\end{array}\right)
\end{displaymath}

As left modules are columns and right modules are rows,
the claims \eqref{prop1-2}--\eqref{prop1-5} follow.

The indecomposable right projective $\mathfrak{C}$-module, 
generated by  $\mathtt{e}_x[i]$, $x\in\mathcal{A}$, is a  direct
summands of the following module $P$:
{\small 
\begin{displaymath}
\left( 
\begin{array}{cccccccccccc}
\dots & 0 & \mathcal{I}_{N-1}
&\mathcal{I}_{N-2}&\dots&\mathcal{I}_{1}&
\mathcal{A} &\mathcal{A}/\mathcal{I}_{N-1} &  
\dots& \mathcal{A}/\mathcal{I}_{1}&0 & \dots\\
\end{array}\right).
\end{displaymath}
}
The filtration \eqref{eq1} induces a filtration on every component of
$P$, whose subquotients could be organized into the following 
rhombal picture:
{\tiny
\begin{equation}\label{eq3}
\begin{array}{ccccccccc}
\mathcal{I}_{N-1}
&\dots&\mathcal{I}_{2}&\mathcal{I}_{1}&
\mathcal{A} &\mathcal{A}/\mathcal{I}_{N-1} &
\mathcal{A}/\mathcal{I}_{N-2} &  
\dots& \mathcal{A}/\mathcal{I}_{1}\\
\hline
\\
&&&&\mathcal{I}_0/\mathcal{I}_{1}&&&&\\
&&&\mathcal{I}_1/\mathcal{I}_{2}&&\mathcal{I}_0/\mathcal{I}_{1}&&&\\
&&\mathcal{I}_2/\mathcal{I}_{3}&&
\mathcal{I}_1/\mathcal{I}_{2}&&\mathcal{I}_0/\mathcal{I}_{1}&&\\
&\dots&\dots&\dots&\dots&\dots&\dots&\dots&\\
\mathcal{I}_{N-1}&\dots&\dots&\dots&\dots&\dots&\dots
&\dots&\mathcal{I}_0/\mathcal{I}_{1}\\
&\dots&\dots&\dots&\dots&\dots&\dots&\dots&\\
&&\mathcal{I}_{N-1}&&
\mathcal{I}_{N-2}/\mathcal{I}_{N-1}&&
\mathcal{I}_{N-3}/\mathcal{I}_{N-2}&&\\
&&&\mathcal{I}_{N-1}&&\mathcal{I}_{N-2}/\mathcal{I}_{N-1}&&&\\
&&&&\mathcal{I}_{N-1}&&&&\\
\end{array}
\end{equation}
}

Organizing these subquotients into a filtration of $P$ as 
shown on the following pictures:
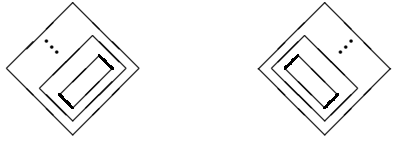
\begin{equation}\label{eq4}
\begin{picture}(60.00,60.00)
\drawline(05.00,30.00)(30.00,55.00)
\drawline(55.00,30.00)(30.00,55.00)
\drawline(55.00,30.00)(30.00,05.00)
\drawline(05.00,30.00)(30.00,05.00)
\drawline(30.00,10.00)(17.50,22.50)
\drawline(37.50,42.50)(17.50,22.50)
\drawline(37.50,42.50)(50.00,30.00)
\drawline(30.00,10.00)(50.00,30.00)
\drawline(30.00,15.00)(25.00,20.00)
\drawline(40.00,35.00)(25.00,20.00)
\drawline(40.00,35.00)(45.00,30.00)
\drawline(30.00,15.00)(45.00,30.00)
\put(20.00,40.00){\makebox(0,0)[cc]{$\cdot$}}
\put(22.00,38.00){\makebox(0,0)[cc]{$\cdot$}}
\put(24.00,36.00){\makebox(0,0)[cc]{$\cdot$}}
\end{picture}
\quad\quad\quad
\begin{picture}(60.00,60.00)
\drawline(05.00,30.00)(30.00,55.00)
\drawline(55.00,30.00)(30.00,55.00)
\drawline(55.00,30.00)(30.00,05.00)
\drawline(05.00,30.00)(30.00,05.00)
\drawline(30.00,10.00)(10.00,30.00)
\drawline(22.50,42.50)(10.00,30.00)
\drawline(22.50,42.50)(42.50,22.50)
\drawline(30.00,10.00)(42.50,22.50)
\drawline(30.00,15.00)(15.00,30.00)
\drawline(20.00,35.00)(15.00,30.00)
\drawline(20.00,35.00)(35.00,20.00)
\drawline(30.00,15.00)(35.00,20.00)
\put(40.00,40.00){\makebox(0,0)[cc]{$\cdot$}}
\put(38.00,38.00){\makebox(0,0)[cc]{$\cdot$}}
\put(36.00,36.00){\makebox(0,0)[cc]{$\cdot$}}
\end{picture}
\end{equation}
we obtain a filtration of $P$ by direct summands of the module
$\Delta_\mathfrak{C}^{1,r}$ and $\Delta_\mathfrak{C}^{2,r}$,
respectively. This means that right $\mathfrak{C}$-projectives 
are filtered by  standard modules for both orders. 
The claim \eqref{prop1-1}  follows and the proof is complete.
\end{proof}

\begin{corollary}\label{cor2}
\begin{enumerate}[(i)]
\item\label{cor2-1}
Left costandard modules for the first  order are given by direct 
summands of the following module:
\begin{displaymath} 
\nabla_\mathfrak{C}^{1,l}=\left( 
\begin{array}{c}
\vdots \\
(\mathcal{I}_2/\mathcal{I}_3)^*\\
( \mathcal{I}_1/\mathcal{I}_2)^*\\
(\mathcal{A}/\mathcal{I}_1)^*\\
0\\
0\\
\vdots
\end{array}\right)
\end{displaymath}
\item\label{cor2-2}
Left costandard modules for the second order are given by direct 
summands of the following module:
\begin{displaymath}\nabla_\mathfrak{C}^{2,l}=\left( 
\begin{array}{c}
\vdots \\
0 \\
0 \\
(\mathcal{A}/\mathcal{I}_1)^*\\
(\mathcal{A}/\mathcal{I}_1)^*\\
(\mathcal{A}/\mathcal{I}_1)^*\\
\vdots
\end{array}\right).
\end{displaymath}
\item\label{cor2-3}
Right costandard modules for the first order are given by direct 
summands of the following module:
\begin{displaymath}\nabla_\mathfrak{C}^{1,r}=\left( 
\begin{array}{ccccccc}
\dots & (\mathcal{A}/\mathcal{I}_1)^* &(\mathcal{A}/\mathcal{I}_1)^*  &(\mathcal{A}/\mathcal{I}_1)^* & 0&0& \dots\\
\end{array}\right)
\end{displaymath}
\item\label{cor2-4}
Right costandard modules for the second order are given by direct 
summands of the following module:
\begin{displaymath} \nabla_\mathfrak{C}^{2,r}=\left( 
\begin{array}{ccccccc}
\dots &0& 0&(\mathcal{A}/\mathcal{I}_1)^* & ( \mathcal{I}_1/\mathcal{I}_2)^*
 & (\mathcal{I}_2/\mathcal{I}_3)^* & \dots\\
\end{array}\right).
\end{displaymath}
\end{enumerate}
\end{corollary}

\begin{proof}
This follows from Proposition~\ref{prop1} applying duality. 
\end{proof}

\begin{corollary}\label{cor25}
For every $x\in \mathcal{A}$ and every $i\in\mathbb{Z}$ there 
is an isomorphism
$\nabla_{\mathfrak{C}}^{2,l}(x,i) \cong \Delta_{\mathfrak{C}}^{1,l}(x,i+N)$.
\end{corollary}

\begin{proof}
Since the $\mathcal{A}$-module $\mathcal{A}/\mathcal{I}_1$ is 
semi-simple by our assumptions, the claim follows directly from Proposition~\ref{prop1}\eqref{prop1-2} and
Corollary~\ref{cor2}\eqref{cor2-2}. 
\end{proof}

Note that, by construction, the original category  $\mathcal{A}$ is
a centralizer subcategory of the category $\mathfrak{C}$.

\section{Symmetric quasi-hereditary envelopes of algebras}\label{s3}

From now on we assume that $\mathcal{A}$ has finitely many objects.
Let $A$ be the path algebra of $\mathcal{A}$. Then
$A$ is a finite-dimensional algebra and we may assume that it is given by
a quiver $Q$ with set of vertices $\{1, \dots, n\}$ and relations $R$. 
As in the previous section, we fix a filtration of $A$ by two-sided ideals 
\begin{equation}\label{eq2}
A=I_0 \supsetneq I_1 \supsetneq I_2\supsetneq \dots \supsetneq I_N=0
\end{equation}
with semisimple subquotients and such that $I_iI_j\subset I_{i+j}$.
For example, we can take \eqref{eq2} to be the radical filtration of $A$.
For $k \in \{1, \dots, n\}$ we denote by $e_k$ the idempotent corresponding to 
the vertex $k$ in $A$, and we denote the corresponding idempotent 
of $A[i]$ (that is in the $(i,i)$th matrix position) by $e_{k,i}$. 
Set $\mathfrak{C}:=\mathfrak{C}(A)$.

\begin{theorem}\label{Csymm}
Assume that $A$ is symmetric with the symmetric trace form $(\cdot,\cdot)$
and that $(\cdot,\cdot)$ induces a non-degenerate pairing between 
$A/I_{j}$ and $I_{N-j}$ for every $j$. Then the algebra
$\mathfrak{C}$ is symmetric.
\end{theorem}

\begin{proof}
Define a bilinear form $(\cdot,\cdot)_{\mathfrak{C}}$ on $\mathfrak{C}$, 
by setting
\begin{displaymath}
(a_{i,j}, b_{k,l})_{\mathfrak{C}}:= \delta_{j,k}\delta_{i,l}(a,b),
\end{displaymath}
where $a,b \in A$ (in a suitable ideal if $i>j$ resp. $k>l$), 
$i,j,k,l \in \mathbb{Z}$, and $a_{i,j}$ means the element $a$ in matrix 
position $(i,j)$. 

The form $(\cdot,\cdot)_{\mathfrak{C}}$ is bilinear, symmetric and
associative by construction. Again, by construction, the form 
$(\cdot,\cdot)_{\mathfrak{C}}$ pairs matrix positions $(i,j)$ and
$(j,i)$. By the  definition of $\mathfrak{C}$, the corresponding 
components in these positions are $A/I_{s}$ and $I_{N-s}$ for some
$s$. By our assumption, the form $(\cdot,\cdot)$ induces a 
nondegenerate pairing of $A/I_{s}$ and $I_{N-s}$. This yields that
$(\cdot,\cdot)_{\mathfrak{C}}$ is nondegenerate as well, completing the
proof.
\end{proof}

\begin{corollary}\label{cor3}
Assume that $A$ is symmetric and that \eqref{eq2} is both the 
radical  and the socle filtration of ${}_A A$ (i.e. ${}_A A$ is 
{\em rigid}). Then $\mathfrak{C}$ is symmetric.
\end{corollary}

\begin{proof}
By our assumptions, the filtration \eqref{eq2} is the unique Loewy
filtration of ${}_A A$. The form $(\cdot,\cdot)$ pairs it with another
Loewy filtration, and hence with itself. This yields that 
$(\cdot,\cdot)$ induces a non-degenerate pairing between 
$A/I_{j}$ and $I_{N-j}$ for every $j$ and the claim follows from
Theorem~\ref{Csymm}.
\end{proof}

Some other examples to which Theorem~\ref{Csymm} can be applied
come from the category $\mathcal{O}$ and will be discussed later on
(see Example~\ref{ex2}). If $A$ is not symmetric (or if it is symmetric 
but does not satisfy the assumptions of Theorem~\ref{Csymm}) it is
more reasonable to try to embed $A$ into its quasi-hereditary ``envelope'' 
not as a centralizer subalgebra, but as an idempotent subquotient. This
goes as follows:

Assume that \eqref{eq2} is the radical filtration of $A$. We form a new 
algebra $\tilde A$ by attaching, for every vertex $k$, a vertex 
$\tilde k$ and an arrow $k \rightarrow \tilde k$, keeping the original
relations $R$, defining the algebra $A$.  Then $A$ is a centralizer
subalgebra of $\tilde A$ (corresponding to nontilded vertices) 
in the natural way, and  $\mathrm{Rad}\tilde A$ 
has nilpotency degree $N+1$. Moreover, the algebra $A$ is also
an idempotent quotient of $\tilde A$, obtained by factoring out the
two-sided ideal, generated by idempotents, associated with the
new (tilded) vertices. Set $\mathbf{N}=\{1,\dots,n\}$,
$\tilde{\mathbf{N}}=\{\tilde 1,\dots,\tilde n\}$, and
$\overline{\mathbf{N}}=\mathbf{N}\cup \tilde{\mathbf{N}}$.

Now $\mathrm{soc} {}_{\tilde A} \tilde A$ consists of simple modules with
indices $\tilde k$. The right projective $e_k\tilde A$ for $\tilde A$,
corresponding to a vertex $k\in \mathbf{N}$,  is the same as the right 
projective for $A$ at the same vertex. The right projective 
$e_{\tilde k}\tilde A$ at vertex $\tilde k\in \tilde{\mathbf{N}}$ 
is an extension of the simple 
at $\tilde k$ with the right projective at $k$ (the simple extending the top 
of $e_kA$), hence has a longer Loewy length. Therefore $e_k\mathrm{Rad}^{N}\tilde A=0$ or, equivalently, the Loewy length $N^r_k$ 
of $e_k\tilde A$ is strictly less than the nilpotency degree of
$\mathrm{Rad}\tilde A$ (which is $N+1$). 
Let $\tilde A e_k$ be the left projective at 
vertex $k$, $N^l_k$ its Loewy length.

We now take $\mathfrak{C}=\mathfrak{C}(\tilde A)$ (with respect to the
radical filtration) and form the  trivial extension
$\mathfrak{D}=\mathfrak{D}(\tilde A)$ of $\mathfrak{C}$ 
with its  ``restricted dual'' $\mathfrak{C}$-bimodule
\begin{displaymath}
\mathfrak{C}^*:= \underset{i,j\in \mathbb{Z};\, x,y\in \overline{\mathbf{N}}}{\bigoplus} 
\mathrm{Hom}_{\Bbbk}(e_{y,j}\mathfrak{C}e_{x,i},\Bbbk). 
\end{displaymath}
(see \cite[3.1]{Ha}). Being a trivial extension of $\mathfrak{C}$, the
algebra $\mathfrak{D}$ is automatically symmetric.
To make the notation consistent with the previous
section, from now on we assume that the nilpotency degree of 
$\mathrm{Rad}\tilde A$ is $N$.

We now extend our first order in the following way: for
$(k,i),(l,j)\in\overline{\mathbf{N}}\times \mathbb{Z}$ we set
$(k,i)>(l,j)$ if $i>j$ or if  $i=j$, $k\in\mathbf{N}$ and 
$l \in \tilde{\mathbf{N}}$. We will again call this order the first order.

\begin{proposition}\label{prop5}
The algebra $\mathfrak{D}(\tilde A)$ is quasi-hereditary with respect to 
the first order and for left standard $\mathfrak{D}$-modules we have
$\Delta_{\mathfrak{D}}^{1,l}(k,i)=\Delta_{\mathfrak{C}}^{1,l}(k,i)$,
$k \in \overline{\mathbf{N}}$, $i\in\mathbb{Z}$.
\end{proposition}

\begin{proof}
We first consider $\mathfrak{C}$. Let $e_{k,i}$ denote the idempotent 
in $\tilde A$ at the vertex $k\in\overline{\mathbf{N}}$, in 
matrix position $i,i$. With respect to our first order 
left standard modules $\Delta_{\mathfrak{C}}^{1,l}(k,i)$ are uniserial 
with a filtration with composition factors 
\begin{displaymath}
L^l(k,i),\, L^l(k,i-1),\, \cdots, \, L^l(k,i-N+1)
\end{displaymath}
read from top to bottom (see Proposition~\ref{prop1}\eqref{prop1-2}).
Then, by \eqref{eq3}, the left projective $\mathfrak{C}e_{k,i}$ 
for $\mathfrak{C}$ has a filtration with subquotients 
\begin{displaymath}
\Delta^{1,l}_{\mathfrak{C}}(k,i),\quad \bigoplus_{j \in J_1} 
\Delta_{\mathfrak{C}}^{1,l}(j,i+1),\quad \dots,\quad 
\bigoplus_{j \in J_{N_k^l}} \Delta_{\mathfrak{C}}^{1,l}(j,i+N_k^l),
\end{displaymath}
where $\mathrm{Rad}^m\tilde Ae_k/\mathrm{Rad}^{m+1}\tilde Ae_k \cong \underset{j \in J_m}{\bigoplus}L^l(j)$.

Similarly, the right projective $e_{k,i}\mathfrak{C}$ has a filtration 
with subquotients 
\begin{displaymath}
\Delta_{\mathfrak{C}}^{2,r}(k,i),\quad
\bigoplus_{j \in \hat J_1} \Delta_{\mathfrak{C}}^{2,r}(j,i-1),\quad 
\dots,\quad 
\bigoplus_{j \in \hat J_{N_k^r}} \Delta_{\mathfrak{C}}^{2,r}(j,i-N_k^r),
\end{displaymath}
where $e_k\mathrm{Rad}^m\tilde A/e_k\mathrm{Rad}^{m+1}\tilde A \cong \underset{j \in \hat J_m}{\bigoplus}L^r(j)$.
Hence the left injective $(e_{k,i}\mathfrak{C})^*$ has a filtration with subquotients 
\begin{displaymath}
\bigoplus_{j \in \hat J_{N_k^r}} \nabla_{\mathfrak{C}}^{2,l}(j,i-N_k^r),\quad
\dots,\quad \bigoplus_{j \in \hat J_1} \nabla_{\mathfrak{C}}^{2,l}(j,i-1),\quad
\nabla_{\mathfrak{C}}^{2,l}(k,i)  
\end{displaymath}
and thus, by the isomorphism $\nabla_{\mathfrak{C}}^{2,l}(k,i) \cong \Delta_{\mathfrak{C}}^{1,l}(k,i+N)$ (Corollary~~\ref{cor25}), 
a filtration with subquotients
\begin{displaymath}
\bigoplus_{j \in \hat J_{N_k^r}} \Delta_{\mathfrak{C}}^{1,l}(j,i+N-N_k^r),\quad
\dots,\quad \bigoplus_{j \in \hat J_1} \Delta_{\mathfrak{C}}^{1,l}(j,i+N-1),\quad
\Delta_{\mathfrak{C}}^{1,l}(k,i+N). 
\end{displaymath}

We now claim that $\mathfrak{D}=\mathfrak{D}(\tilde A)$ is 
quasi-hereditary with 
$\Delta_{\mathfrak{D}}^{1,l}(k,i)=\Delta_{\mathfrak{C}}^{1,l}(k,i)$. 
As the projective module $\mathfrak{D}e_{k,i}$ has a filtration with 
subquotients $\mathfrak{C}e_{k,i}$ and $(e_{k,i}\mathfrak{C})^*$, which both 
have $\Delta_{\mathfrak{D}}^{1,l}$-filtrations by above,  
$\mathfrak{D}e_{k,i}$ also has  a $\Delta_{\mathfrak{D}}^{1,l}$-filtration. 
So it suffices to check that all 
standard modules appearing in $(e_{k,i}\mathfrak{C})^*$ have larger index than 
$(k,i)$. To see this, we need to distinguish two cases. 

The first case is when $k \in \mathbf{N}$. In this case, the 
smallest second index of the  standard modules appearing in
$(e_{k,i}\mathfrak{C})^*$ is $i+N-N_k^r$. But, as  seen above, for 
$k \in \{1, \dots, n\}$, $N_k^r <N$, so $i+N-N_k^r > i$, which is what we need.

The second case is when $k \in \tilde{\mathbf{N}}$. In this case 
the smallest second index of the standard modules appearing in 
$(e_{k,i}\mathfrak{C})^*$ can well be $i$, however, in this case $P^r(k)$ has 
simple top $L^r(k)$ and all other composition factors are of the form $L^r(j)$, 
with $j \in  \{1, \dots, n\}$. Therefore, the standard modules appearing in 
$(e_{k,i}\mathfrak{C})^*$ with smallest second index, namely 
$\Delta_{\mathfrak{C}}^{1,l}(j,i)$, have first index $j$ where $L^r(j)$ 
occurs in $e_k\mathrm{Rad}^{N_k^r}\tilde A $, so $j \in\mathbf{N}$, 
and $(k,i) < (j,i)$. This completes the proof that $\mathfrak{D}$ is 
quasi-hereditary.
\end{proof}

From Proposition~\ref{prop5} and  \cite[Corollary~5]{MT1} it follows
that, with respect to the first order, right $\mathfrak{D}$-projectives
also have standard filtrations. The corresponding standard modules are
described as follows:

\begin{lemma}\label{lem6}
The right standard module 
$\Delta_{\mathfrak{D}}^{1,r}(i,k)$ for $\mathfrak{D}$ is an extension of the $\mathfrak{C}$-modules $\Delta_{\mathfrak{C}}^{1,r}(i,k)$ and $\nabla_{\mathfrak{C}}^{2,r}(i-N+1,k)$.
\end{lemma}

\begin{proof}
The right projective module $e_{k,i}\mathfrak{D}$ has a filtration with 
subquotients $e_{k,i}\mathfrak{C}$ and $(\mathfrak{C}e_{k,i})^*$.
The module $e_{k,i}\mathfrak{C}$ is filtered by 
\begin{displaymath}
\Delta_{\mathfrak{C}}^{1,r}(k,i),\, \Delta_{\mathfrak{C}}^{1,r}(k,i+1),\, 
\dots,\,  \Delta_{\mathfrak{C}}^{1,r}(k,i+N-1)
\end{displaymath}
and the module $\mathfrak{C}e_{k,i}$ is filtered by 
\begin{displaymath}
\Delta_{\mathfrak{C}}^{2,l}(k,i),\, \Delta_{\mathfrak{C}}^{2,l}(k,i-1),\, 
\dots,\,  \Delta_{\mathfrak{C}}^{2,l}(k,i-N+1)
\end{displaymath}
Therefore the module $(\mathfrak{C}e_{k,i})^*$ is filtered by \begin{displaymath}
\nabla_{\mathfrak{C}}^{2,r}(k,i-N+1),\, \nabla_{\mathfrak{C}}^{2,r}(k,i-N+2),\, \dots,\, \nabla_{\mathfrak{C}}^{2,r}(k,i).
\end{displaymath}

Let $X$ denote the quotient of $e_{k,i}\mathfrak{D}$ modulo the trace
of all $e_{k,j}\mathfrak{D}$, $j>i$. Obviously 
$\Delta_{\mathfrak{C}}^{1,r}(k,i)$ is a quotient of $X$. Since none of 
modules $e_{k,j}\mathfrak{D}$, $j>i$, contains $L^r(k,i-N+1)$, 
$\nabla_{\mathfrak{C}}^{2,r}(k,i-N+1)$ is a subquotient of $X$ as well. 
By definition, none of other $\Delta_{\mathfrak{C}}^{1,r}(k,j)$ 
contributes to $X$, which yields that $X$ has a quotient $\tilde X$, which 
is an extension of $\Delta_{\mathfrak{C}}^{1,r}(k,i)$ 
by $\nabla_{\mathfrak{C}}^{2,r}(k,i-N+1)$. 

As $\mathfrak{C}$ is quasi-hereditary with respect to the second order,
we also have a quotient $\Delta_{\mathfrak{C}}^{2,r}(k,i)$ which is uniserial 
with a filtration $L^r(k,i), L^r(k,i+1),\dots, L^r(k,i+N-1)$. Since 
$L^r(k,i+1)$ is in the top of the kernel of $e_{k,i}\mathfrak{D} 
\twoheadrightarrow X$, we know that $\Delta^{1,r}_\mathfrak{D} (k,i+1)$ 
appears as a subquotient of a standard filtration of 
$e_{k,i}\mathfrak{D}$. Inductively, we  obtain that the modules
\begin{displaymath}
\Delta^{1,r}_\mathfrak{D} (k,i),\, \Delta^{1,r}_\mathfrak{D} (k,i+1),\, 
\dots,\,  \Delta^{1,r}_\mathfrak{D} (k,i+N-1) 
\end{displaymath}
appear as subquotients  of a standard filtrations of $e_{k,i}\mathfrak{D}$. 
Each of those $\Delta^{1,r}_\mathfrak{D} (k,j)$ has a quotient which is an 
extension of $\Delta_{\mathfrak{C}}^{1,r}(k,j)$ by 
$\nabla_{\mathfrak{C}}^{2,r}(k,j-N+1)$ and we see that this exhausts the 
whole module. Hence the surjection of $X$  onto $\tilde X$ must be an 
isomorphism and right standard modules in the first 
order for $\mathfrak{D}$ are of the desired form.
\end{proof}

\begin{corollary}\label{cor8}
The algebra $\mathfrak{D}$ is quasi-hereditary with respect to the
second order as well.
\end{corollary}

\begin{proof}
Since $\mathfrak{D}$ is symmetric, projective and injective 
$\mathfrak{D}$-modules coincide. By Proposition~\ref{prop5},
left standard $\mathfrak{D}$-modules with respect to the first order
are uniserial and coincide with the corresponding $\mathfrak{C}$-modules.
Take a standard filtration of a left projective $\mathfrak{D}$-module.
Applying duality we get a costandard filtration of a right 
injective $\mathfrak{D}$-module, which is also a right projective.

Taking into account that these right costandard modules coincide, up 
to shift, with right standard modules with respect to the second order,
we obtain that right projective $\mathfrak{D}$-modules have a filtration
by right standard modules with respect to the second order. Since
all shifts are the same (by $N$), it follows that this filtration
satisfies the necessary ordering condition. The claim follows.
\end{proof}

\begin{remark}\label{rem10}
{\rm 
Assume that the right projective $\tilde A$-module at vertex $k$ has 
radical filtration with subquotients $P^1, \dots P^s$. Then the
indecomposable right projective $\mathfrak{C}$-module at $(k,i)$ 
looks as follows:
\begin{equation}\label{eq5}
\xymatrix@!=0pt{
&&&&&P_i^1&&&&\\
&&&&P_{i+1}^1&&P_{i-1}^2&&&\\
&&&P_{i+2}^1&&P_i^2&&\ddots&&\\
&&&&\ddots&&\ddots&&P^s_{i-s+1}&\\
&&&&&&&P^s_{i-s+2}&&\\
P^1_{i+N-1}  &&&&&&&&&\\
&P^2_{i+N-2}&&&&&&&&\\
&&&&&&&&&\\
&&&P^s_{i+N-s}&&&&&&\\
&&&&&&&&&
}
\end{equation}
If the left projective $\tilde A$-module  at vertex $k$ has 
a radical filtration with subquotients $Q^1, \dots Q^t$, then the
indecomposable right injective $\mathfrak{C}$-module at $(k,i)$ looks 
as follows:
\begin{equation}\label{eq6}
\xymatrix@!=2pt{
&&&&&Q_{i-N+t}^t&&&&\\
&&&&Q_{i-N+t+1}^t&&Q_{i-N+t-1}^{t-1}&&&\\
&&&Q_{i-N+t+2}^t&&Q_{i-N+t}^{t-1}&&\ddots&&\\
&&&&\ddots&&\ddots&&Q^1_{i-t+1}&\\
&&&&&&&Q^1_{i-t+2}&&\\
Q^t_{i+t-1}  &&&&&&&&&\\
&Q^{t-1}_{i+t-2}&&&&&&&&\\
&&&&&&&&&\\
&&&Q^1_{i}&&&&&&\\
&&&&&&&&&
}
\end{equation}
For $ k \in \tilde{\mathbf{N}}$, $s$ can reach $N$, but $t=1$.
For $k \in\mathbf{N}$, $t$ can reach $N$, but $s$ is always less than $N$ 
and $Q^t$ only has composition factors indexed by  $k \in \mathbf{N}$.

The corresponding indecomposable injective $\mathfrak{D}$-module is
obtained by gluing \eqref{eq5} and \eqref{eq6}. The standard filtrations
of this module with respect to the first and the second order can be
organized using the left and the right diagrams from \eqref{eq4},
respectively.
}
\end{remark}

\begin{proposition}\label{prop9}
A is an idempotent subquotient of $\mathfrak{D}$ as follows:
\begin{displaymath}
A \cong 1_{\tilde A_i}\mathfrak{D} 1_{\tilde A_i}/
1_{\tilde A_i}\mathfrak{D}\tilde e_i\mathfrak{D}1_{\tilde A_i},
\quad\text{ where }\quad
\tilde e_i= \sum_{k=1}^n e_{\tilde k, i} \in \mathfrak{D}.
\end{displaymath}
\end{proposition}

\begin{proof}
It is obvious that $1_{\tilde A_i}\mathfrak{D} 1_{\tilde A_i}$ is isomorphic 
to the trivial extension $S$ of $\tilde A$ by ${\tilde A}^*$. Now we 
claim  that the ideal ${\tilde A}^*$ is contained in the ideal, 
generated by $\displaystyle\tilde e:=\sum_{k=1}^n e_{\tilde k}$.

Consider the right projective $S$-module $e_kS$ at vertex 
$k \in \mathbf{N}$. This module has a filtration by the right projective 
$\tilde A$-module $e_k\tilde A$  (which is the dual of the corresponding
left injective  $A$-module and only has composition factors $L^r(j)$ 
for  $j  \in \mathbf{N}$), and the right injective $\tilde A$-module  at the 
vertex $k$, which has a semisimple quotient consisting of simples $L^r(r)$ 
for $r \in \tilde{\mathbf{N}}$ and sitting on a submodule isomorphic to 
the right injective $A$-module at the vertex $k$. Hence, right projectives 
for $S/S\tilde e S$ look like right projectives for $A$, so 
$S/S\tilde e S \cong A$. 
\end{proof}

\section{Triangular decomposition}\label{s4}

Recall (see \cite{Ko,Ko2}) that a directed subalgebra $B$ of a 
basic quasi-hereditary
algebra $A$ is called a {\em (strong) exact Borel subalgebra} provided 
that $A$ and $B$ have the same simple modules, the tensor induction 
functor $A\otimes_B{}_-$  is exact and maps simple modules to standard 
modules. Dually one defines {\em (strong) $\Delta$-subalgebras} (again 
see \cite{Ko}). There is an obvious generalization of these notions
to $\Bbbk$-linear categories (our algebras). We keep the setup of the 
previous section and identify the algebra $A/{I}_1$ with some 
maximal semisimple subalgebra of $A$, say $S$. Then $S$ is a
maximal semisimple subalgebra (in particular, a subspace) of all 
algebras $A/I_i$ for all $i>0$.

\begin{proposition}\label{prop21}
The algebra
\begin{displaymath}
\tilde{\mathcal{B}}:=\left( 
\begin{array}{cccccccc}
\ddots & \vdots & \ddots & \vdots & \vdots &\vdots & \vdots & \ddots\\
\dots & S & S & \ddots &  S & 0 &0 & \dots\\
\dots &0  & S & S & \ddots &  S & 0 & \dots\\
\dots &0  & 0 & S  & S  & \ddots & S & \dots\\
\dots &0&0&0&S &S & \ddots & \dots\\
\dots &0 & 0 &  0  &0  & S & S & \ddots\\
\dots &0 & 0 &  0  &0  &0  & S & \dots\\
\ddots & \vdots & \vdots & \vdots &\vdots &\vdots & \vdots & \ddots\\
\end{array}
\right)
\end{displaymath}
(here each row contains exactly $N$ nonzero entries) is a strong 
exact $\Delta$-subalgebra of both $\mathfrak{C}$ and $\mathfrak{D}$
with respect to the first order.
\end{proposition}

\begin{proof}
The algebra $\tilde{\mathcal{B}}$ is obviously a subalgebra of 
both $\mathfrak{C}$ and $\mathfrak{D}$. It is directed by definition
and thus quasi-hereditary with respect to the first order. Corresponding
right standard modules are just simple modules, corresponding
left standard modules are projectives and look as follows: 
\begin{displaymath}
\left( \begin{array}{c}
\vdots \\
{A}/{I}_1 \\
{A}/{I}_1\\
{A}/{I}_1\\
0 \\
0 \\
\vdots \\
\end{array}
\right)
\end{displaymath}
These coincide with left standard modules for both $\mathfrak{C}$ 
and $\mathfrak{D}$ (by Proposition~\ref{prop1}\eqref{prop1-2}
and Proposition~\ref{prop5}). Therefore, using \cite[Theorem~A]{Ko}, 
we deduce that $\tilde{\mathcal{B}}^{op}$ is an exact Borel subalgebra 
for  $\mathfrak{C}^{op}$ and $\mathfrak{D}^{op}$. Thus, by 
\cite[Theorem~B]{Ko}, we have that $\tilde{\mathcal{B}}$ is 
a $\Delta$-subalgebra for $\mathfrak{C}$ and $\mathfrak{D}$. That
$\tilde{\mathcal{B}}$ is strong follows from the definitions. 
This completes the proof.
\end{proof}

Assume now that the algebra $A$ is positively graded,
$\displaystyle A=\bigoplus_{i=0}^{\infty}A_i$ and that 
the filtration \eqref{eq2} coincides with the grading filtration, that is
$\displaystyle  I_j=\bigoplus_{i=j}^{\infty}A_i$. In this
case we have $I_j/I_{j+1}\cong A_j$ for all $i$, in particular, 
$I_j/I_{j+1}$ can be realized as a canonical subspace of $A$. 

\begin{proposition}\label{prop22} 
Under the above assumptions, the algebra 
\begin{displaymath}
\mathcal{B}:=\left( 
\begin{array}{cccccc}
\ddots & \vdots & \vdots & \vdots & \vdots & \ddots\\
\dots & A_0 & 0 & 0 & 0 & \dots\\
\dots & A_1 &A_0  &0  &0 & \dots\\
\dots & A_2& A_1&A_0 &0 & \dots\\
\dots & A_3 & A_2 & A_1 & A_0 & \dots\\
\ddots & \vdots & \vdots & \vdots & \vdots & \ddots\\
\end{array}
\right)
\end{displaymath}
is a strong exact Borel subalgebra of $\mathfrak{C}$ with respect
to the first order.
\end{proposition}

\begin{proof}
That $\mathcal{B}$ is a subalgebra follows from the definitions and the
fact that $A$ is graded (i.e. $A_iA_j\subset A_{i+j}$). Note that 
$A_0$ is a maximal semi-simple subalgebra of $A$ and hence simple
$A$-modules can be identified with simple $A_0$-modules. Therefore
simple $\mathfrak{C}$-modules (shifted simple $A$-modules) 
and $\mathcal{B}$-modules (shifted simple $A_0$-modules) can be 
identified as well. 

The algebra $\mathcal{B}$ is directed by definition
hence quasi-hereditary with respect to the first order.
Left standard $\mathcal{B}$-modules are simple. Right standard  
$\mathcal{B}$-modules are projective. Left costandard 
$\mathcal{B}$-modules are dual to right standard  
$\mathcal{B}$-modules and hence have the following form:
\begin{displaymath}
\left( 
\begin{array}{c}
 \vdots \\
A_2^*\\
A_i^*\\ 
A_0^* \\
 0 \\
 0 \\
\vdots\\
\end{array}\right)
\end{displaymath}
As $A_j\cong I_j/I_{j+1}$ for all $j$, from
Corollary~\ref{cor2}\eqref{cor2-1} we obtain that these 
costandard modules are restrictions of costandard 
$\mathfrak{C}$-modules. Hence $\mathcal{B}$ is an exact Borel 
subalgebra by \cite[Theorem~A]{Ko}. That $\mathcal{B}$ is strong
follows from the definitions. This completes the proof.
\end{proof}

\begin{remark}
{\rm If we assume the existence of a Borel subalgebra, the condition of left costandard modules for this algebra being isomorphic to left costandard modules for $\mathfrak{C}$ forces the Borel subalgebra to have the following form:
\begin{displaymath}
\left( 
\begin{array}{cccccc}
\ddots & \vdots & \vdots & \vdots & \vdots & \ddots\\
\dots & X_{0,i-1} & 0 & 0 & 0 & \dots\\
\dots & X_{1,i-1} &X_{0,i}  &0  &0 & \dots\\
\dots & X_{2,i-1}& X_{1,i}&X_{0,i+1} &0 & \dots\\
\dots & X_{3,i-1} & X_{2,i} & X_{1,i+1} & X_{0,i+2} & \dots\\
\ddots & \vdots & \vdots & \vdots & \vdots & \ddots\\
\end{array}
\right)
\end{displaymath}
where $ X_{j,i}$ are subspaces of $I_j$ providing a splitting of $I_j \twoheadrightarrow I_j/I_{j+1}$. Furthermore we must have $ X_{j,i} X_{i-k,k} \subseteq X_{j+i-k,k}$ for this to be a subalgebra.
If we assume that the Borel subalgebra is stable under the shift, i.e. that 
$ X_{j,i} = X_{j,i+1}$ for all $i,j$, then the above is simply the condition that $A$ is graded. Hence the existence of a Borel subalgebra which is invariant under the shift is equivalent to $A$ being graded with respect to the filtration (\ref{eq2}).
}
\end{remark}

We further assume that $A$ is positively graded. Then the trivial 
extension $\overline{A}=A\oplus A^*$ of $A$ inherits a natural
$\mathbb{Z}$-grading by assigning degree $-i$ to the space 
$A_i^*$, $i\geq 0$. We would need to redefine this natural grading as 
follows: set $\mathrm{deg}A_i^*=N-1-i$. For $i\in\mathbb{Z}$ set 
$\overline{A}_i=A_i\oplus A^*_{N-1-i}$ and, because of
$\mathrm{Rad}^N(A)=0$,  we have  $\overline{A}_i=0$ for all $i<0$.

\begin{proposition}\label{prop23} 
Under the assumptions of Proposition~\ref{prop22}, the algebra 
\begin{displaymath}
\overline{\mathcal{B}}:=\left( 
\begin{array}{cccccc}
\ddots & \vdots & \vdots & \vdots & \vdots & \ddots\\
\dots & \overline{A}_0 & 0 & 0 & 0 & \dots\\
\dots & \overline{A}_1 & \overline{A}_0 & 0  &0 & \dots\\
\dots & \overline{A}_2 & \overline{A}_1 & \overline{A}_0 & 0 & \dots\\
\dots & \overline{A}_3 & \overline{A}_2 & \overline{A}_1 & 
\overline{A}_0 & \dots\\
\ddots & \vdots & \vdots & \vdots & \vdots & \ddots\\
\end{array}
\right)
\end{displaymath}
is a strong exact Borel subalgebra of $\mathfrak{D}$ with respect
to the first order.
\end{proposition}

\begin{proof}
That $\overline{\mathcal{B}}$ is a directed subalgebra of 
$\mathfrak{D}$ and that simple $\overline{\mathcal{B}}$ and
$\mathfrak{D}$ modules can be identified follows from the
construction. Using Lemma~\ref{lem6}, the rest is proved just as in 
the proof  of Proposition~\ref{prop22}.
\end{proof}

Denote by $\mathtt{S}_{\mathbb{Z}}$ the subalgebra 
\begin{displaymath}
\tilde{\mathcal{B}}:=\left( 
\begin{array}{ccccc}
\ddots & \vdots & \vdots & \vdots & \ddots\\
\dots & A_0 & 0 & 0 &  \dots\\
\dots & 0 & A_0 & 0 &  \dots\\
\dots & 0 & 0 & A_0 &  \dots\\
\ddots &\vdots &\vdots & \vdots & \ddots\\
\end{array}
\right)
\end{displaymath}
of $\mathfrak{C}$. Note that $\mathtt{S}_{\mathbb{Z}}$ is a semi-simple
subalgebra of $\mathfrak{D}$, $\tilde{\mathcal{B}}$,   $\mathcal{B}$
and $\overline{\mathcal{B}}$.
Propositions~\ref{prop22} and \ref{prop23} allow 
us to deduce the following {\em triangular decompositions} for the algebras 
$\mathfrak{C}$ and $\mathfrak{D}$:

\begin{theorem}\label{thm25}
Under the assumptions of Proposition~\ref{prop22} we have:
\begin{enumerate}
\item\label{thm25-1}
Multiplication in $\mathfrak{C}$ induce the following isomorphism
of left $\tilde{\mathcal{B}}$- and right $\mathcal{B}$-modules:
$\mathfrak{C}\cong \tilde{\mathcal{B}} 
\otimes _{\mathtt{S}_{\mathbb{Z}}} \mathcal{B}$.
\item\label{thm25-2}
Multiplication in $\mathfrak{D}$ induce the following isomorphism
of left $\tilde{\mathcal{B}}$- and right $\overline{\mathcal{B}}$-modules:
$\mathfrak{D}\cong \tilde{\mathcal{B}} 
\otimes _{\mathtt{S}_{\mathbb{Z}}} \overline{\mathcal{B}}$.
\end{enumerate}
\end{theorem}

\begin{proof}
This follows from  Propositions~\ref{prop22} and \ref{prop23}
and \cite{Ko2}.
\end{proof}

Similarly one obtains the following:

\begin{theorem}\label{thm26}
With respect to the second order we have the following:
\begin{enumerate}
\item\label{thm26-1}
The algebra $\tilde{\mathcal{B}}$ is a strong exact Borel subalgebra
of both $\mathfrak{C}$ and $\mathfrak{D}$.
\item\label{thm26-2}
Under the assumptions of Proposition~\ref{prop22}, 
the algebra $\mathcal{B}$ is a strong exact $\Delta$-subalgebra
of $\mathfrak{C}$.
\item\label{thm26-3}
Under the assumptions of Proposition~\ref{prop22}, 
the algebra $\overline{\mathcal{B}}$ is a strong exact
$\Delta$-subalgebra of $\mathfrak{D}$.
\end{enumerate}
\end{theorem}

\begin{proof}
Left to the reader. 
\end{proof}

\begin{corollary}\label{cor2601}
Under the assumptions of Proposition~\ref{prop22}, we have that 
$A\mathrm{-mod}$ embeds into $\mathcal{F}(\Delta^{1,l}_{\mathfrak{C}})$.
\end{corollary}

\begin{proof}
As $\mathcal{B}$ is a Borel subalgebra of $\mathfrak{C}$, we have that 
$\mathcal{B}\mathrm{-mod}$ embeds into 
$\mathcal{F}(\Delta^{1,l}_{\mathfrak{C}})$
via exact tensor induction. As $A$ is an idempotent subquotient
of $\mathcal{B}$ by construction, the claim follows.
\end{proof}

Similarly we have the following:

\begin{corollary}\label{cor26}
Under the assumptions of Proposition~\ref{prop22}, we have that 
$\mathrm{mod-}A$ embeds into $\mathcal{F}(\Delta^{2,r}_{\mathfrak{C}})$.
\end{corollary}

Let $B$ be the path algebra of the quiver
\begin{displaymath}
\xymatrix{ 
\dots\ar[r]&\bullet\ar[r]&\bullet\ar[r]&\bullet\ar[r]&\bullet\ar[r]&\dots
}
\end{displaymath}
modulo the relations that any composition of $N$ arrows is zero.

\begin{corollary}\label{cor27}
The category  $B\mathrm{-mod}$ embeds into 
$\mathcal{F}(\Delta^{2,l}_{\mathfrak{C}})$.
\end{corollary}

\begin{proof}
The algebra $\tilde B$ consists of direct summands, each of which 
is isomorphic to $B$. As $\tilde B$ is a $\Delta$-subalgebra of 
$\mathfrak{C}$, we have that $\tilde B\mathrm{-mod}$, and 
hence $B\mathrm{-mod}$, embeds into $\mathcal{F}(\nabla_\mathfrak{C}^{1,l})$. However, up to a shift, costandard modules in the first order are the same 
as standard modules in the second order by Corollary~\ref{cor25}, so $\mathcal{F}(\nabla^{1,l}_\mathfrak{C})=
\mathcal{F}(\Delta^{2,l}_\mathfrak{C})$. This completes the proof.
\end{proof}

Similarly we have:

\begin{corollary}\label{cor28}
The category $\mathrm{mod-}B$ embeds into 
$\mathcal{F}(\Delta^{1,r}_{\mathfrak{C}})$.
\end{corollary}

\begin{corollary}\label{cor29}
\begin{enumerate}
\item\label{cor29-1} 
The category $\mathrm{mod-}B$ embeds into 
$\mathcal{F}(\Delta^{1,r}_{\mathfrak{D}})$.
\item\label{cor29-2}
The category  $B\mathrm{-mod}$ embeds into 
$\mathcal{F}(\Delta^{2,l}_{\mathfrak{D}})$.
\item\label{cor29-3}
Under the assumptions of Proposition~\ref{prop22}, we have 
that  $\overline{A}\mathrm{-mod}$ embeds into 
$\mathcal{F}(\Delta^{1,l}_{\mathfrak{D}})$.
\item\label{cor29-4}
Under the assumptions of Proposition~\ref{prop22}, we have 
that  $\mathrm{mod-}\overline{A}$ embeds into 
$\mathcal{F}(\Delta^{2,r}_{\mathfrak{D}})$.
\end{enumerate}
\end{corollary}

\section{Examples}\label{s5}

\begin{example}[An easy quiver algebra]\label{ex3}
{\rm
Let $A$ be the path algebra of the following quiver:
\begin{displaymath}
\xymatrix{ 
\mathtt{1}\ar[rr]^{\mathtt{a}}&&\mathtt{2}.
}
\end{displaymath}
Assume that \eqref{eq2} is the radical filtration of $A$.
Let $\mathtt{e}_1$ and $\mathtt{e}_2$ be the idempotents of $A$,
corresponding to the vertices $\mathtt{1}$ and $\mathtt{2}$, respectively.
In this case the algebra $\mathfrak{C}(A)$ is the path algebra 
of the following quiver:
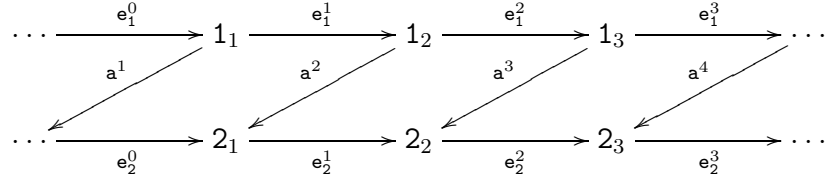
\begin{displaymath}
\xymatrix{ 
\dots\ar[rr]^{\mathtt{e}^{0}_\mathtt{1}} && 
\mathtt{1}_{1}\ar[rr]^{\mathtt{e}^{1}_\mathtt{1}}\ar[dll]_{\mathtt{a}^1}
&& \mathtt{1}_2\ar[rr]^{\mathtt{e}^{2}_\mathtt{1}}\ar[dll]_{\mathtt{a}^2}
&& \mathtt{1}_3\ar[rr]^{\mathtt{e}^{3}_\mathtt{1}}\ar[dll]_{\mathtt{a}^3} 
&& \dots\ar[dll]_{\mathtt{a}^4}\\
\dots\ar[rr]_{\mathtt{e}^{0}_\mathtt{2}} && 
\mathtt{2}_{1}\ar[rr]_{\mathtt{e}^{1}_\mathtt{2}}
&& \mathtt{2}_2\ar[rr]_{\mathtt{e}^{2}_\mathtt{2}}
&& \mathtt{2}_3\ar[rr]_{\mathtt{e}^{3}_\mathtt{2}} && \dots\\
}
\end{displaymath}
modulo the ideal, generated by the following relations:
\begin{equation}\label{eq7}
\mathtt{e}^{i+1}_\mathtt{1}\mathtt{e}^{i}_\mathtt{1}=
\mathtt{e}^{i+1}_\mathtt{2}\mathtt{e}^{i}_\mathtt{2}=0,\quad
\mathtt{e}^{i-1}_\mathtt{2}\mathtt{a}^i=
\mathtt{a}^{i+1}\mathtt{e}^{i}_\mathtt{1},
\end{equation}
where $i\in\mathbb{Z}$. 

We also have  $A\cong\tilde{\Bbbk}$ (where $\mathtt{2}=\tilde{\mathtt{1}}$).
In this case the algebra $\mathfrak{D}(\tilde{\Bbbk})$ is the path algebra 
of the following quiver (the dual part $\mathfrak{C}^*$ is depicted using
the dotted arrows):
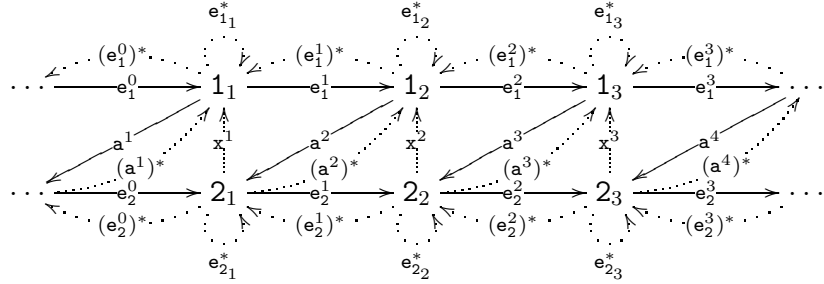
\begin{displaymath}
\xymatrix{ 
\dots\ar[rr]|-{\mathtt{e}^{0}_\mathtt{1}} && 
\mathtt{1}_{1}\ar[rr]|-{\mathtt{e}^{1}_\mathtt{1}}\ar[dll]|-{\mathtt{a}^1}
\ar@(ul,ur)@{.>}[]^{\mathtt{e}_{\mathtt{1}_{1}}^*}
\ar@/_1pc/@{.>}[ll]|-{(\mathtt{e}^{0}_\mathtt{1})^*}
&& \mathtt{1}_2\ar[rr]|-{\mathtt{e}^{2}_\mathtt{1}}\ar[dll]|-{\mathtt{a}^2}
\ar@(ul,ur)@{.>}[]^{\mathtt{e}_{\mathtt{1}_{2}}^*}
\ar@/_1pc/@{.>}[ll]|-{(\mathtt{e}^{1}_\mathtt{1})^*}
&& \mathtt{1}_3\ar[rr]|-{\mathtt{e}^{3}_\mathtt{1}}\ar[dll]|-{\mathtt{a}^3} 
\ar@(ul,ur)@{.>}[]^{\mathtt{e}_{\mathtt{1}_{3}}^*}
\ar@/_1pc/@{.>}[ll]|-{(\mathtt{e}^{2}_\mathtt{1})^*}
&& \dots\ar[dll]|-{\mathtt{a}^4}
\ar@/_1pc/@{.>}[ll]|-{(\mathtt{e}^{3}_\mathtt{1})^*}
\\
\dots\ar[rr]|-{\mathtt{e}^{0}_\mathtt{2}}
\ar@/_1pc/@{.>}[rru]|-{(\mathtt{a}^1)^*}&& 
\mathtt{2}_{1}\ar[rr]|-{\mathtt{e}^{1}_\mathtt{2}}
\ar@(dl,dr)@{.>}[]_{\mathtt{e}_{\mathtt{2}_{1}}^*}
\ar@/^1pc/@{.>}[ll]|-{(\mathtt{e}^{0}_\mathtt{2})^*}
\ar@/_1pc/@{.>}[rru]|-{(\mathtt{a}^2)^*}
\ar@{.>}[u]|-{\mathtt{x}^1}
&& \mathtt{2}_2\ar[rr]|-{\mathtt{e}^{2}_\mathtt{2}}
\ar@(dl,dr)@{.>}[]_{\mathtt{e}_{\mathtt{2}_{2}}^*}
\ar@/^1pc/@{.>}[ll]|-{(\mathtt{e}^{1}_\mathtt{2})^*}
\ar@/_1pc/@{.>}[rru]|-{(\mathtt{a}^3)^*}
\ar@{.>}[u]|-{\mathtt{x}^2}
&& \mathtt{2}_3\ar[rr]|-{\mathtt{e}^{3}_\mathtt{2}} 
\ar@(dl,dr)@{.>}[]_{\mathtt{e}_{\mathtt{2}_{3}}^*}
\ar@/^1pc/@{.>}[ll]|-{(\mathtt{e}^{2}_\mathtt{2})^*}
\ar@/_1pc/@{.>}[rru]|-{(\mathtt{a}^4)^*}
\ar@{.>}[u]|-{\mathtt{x}^3}
&&\ar@/^1pc/@{.>}[ll]|-{(\mathtt{e}^{3}_\mathtt{2})^*} \dots\\
}
\end{displaymath}
(here $x^i=(\mathtt{e}^{i-1}_\mathtt{2}\mathtt{a}^i)^*$)
modulo the ideal, generated by the relations \eqref{eq7},
the relations saying that the product of any two dotted arrows 
is zero, and the relations defining the natural 
$\mathfrak{C}$-bimodule  structure on $\mathfrak{C}^*$.
}
\end{example}

\begin{example}[Schur algebras for $GL_2$]\label{ex1}
{\rm

Let $A$ be a block of a Schur algebras for $GL_2$, say with $ap^k+r$ simple modules ($1\leq a \leq p-1, k\geq 0 , 1\leq r \leq p^k$). These have been extensively studied in \cite{MT1} and \cite{MT2} and in particular have been shown to be a hereditary idempotent subquotients of an infinite-dimensional symmetric quasi-hereditary algebra. Instead of taking an idempotent subquotient, one might also take a centralizer subalgebra $B$ which is again symmetric, such that it corresponds to  the endomorphism ring of the first $ap^k$ projectives for the Schur algebra. From the explicit description in terms of quivers and relations in \cite{MT2}, it is easily seen that this has a $\mathbb{Z}$-grading, which coincides with the radical filtration, hence has semisimple subquotients. By \cite[Theorem~3.3]{MV}, any connected
finite-dimensional self-injective positively graded algebra is rigid. 
Therefore we can apply Corollary~\ref{cor3} to obtain a symmetric quasi-hereditary algebra. This will however give an algebra that is significantly larger than the symmetric quasihereditary envelope given in \cite{MT1,MT2}.
}
\end{example}

\begin{example}[Category $\mathcal{O}$]\label{ex2}
{\rm
Let $\mathfrak{g}$ be a semi-simple finite-dimensional complex Lie
algebra with a fixed triangular decomposition 
$\mathfrak{g}=\mathfrak{n}_-\oplus \mathfrak{h}\oplus \mathfrak{n}_+$,
and $\mathfrak{p}\supset \mathfrak{h}\oplus \mathfrak{n}_+$ be a 
parabolic subalgebra of $\mathfrak{g}$. Let $\mathcal{O}_0^{\mathfrak{p}}$
denote the principal block of the $\mathfrak{p}$-parabolic category
$\mathcal{O}$ for $\mathfrak{g}$, and $A^{\mathfrak{p}}$ denote the
endomorphism algebra of the multiplicity-free direct sum  of all 
indecomposable projective-injective modules in $\mathcal{O}_0^{\mathfrak{p}}$.

The algebra $A^{\mathfrak{p}}$ is positively graded and symmetric
(see \cite{MS}) and simple $A^{\mathfrak{p}}$-modules are naturally 
indexed by the elements of some right cell for the Weyl group $W$ of
$\mathfrak{g}$. In the special case $\mathfrak{g}=\mathfrak{sl}_n$, the
parabolic subalgebra $\mathfrak{p}$ is given by some composition of
$n$ and the  algebra $A^{\mathfrak{p}}$ can be used to model the
corresponding Specht module (for the symmetric group or Hecke algebra)
via the action of some exact functors on $A^{\mathfrak{p}}\mathrm{-mod}$,
see \cite{KMS}. The algebra $A^{\mathfrak{p}}$ has a simple preserving
duality, which yields that all indecomposable projective
$A^{\mathfrak{p}}$-modules are self-dual. Since the trace form on
$A^{\mathfrak{p}}$ respects grading, it follows that this form induces 
a nondegenerate pairing between the components of the grading 
filtration of $A^{\mathfrak{p}}$ as required in the formulation of
Theorem~\ref{Csymm}. Thus, from Theorem~\ref{Csymm} it follows that 
the quasi-hereditary envelope $\mathfrak{C}(A^{\mathfrak{p}})$ of
$A^{\mathfrak{p}}$ is symmetric and thus $A^{\mathfrak{p}}$ is a 
centralizer subalgebra of a symmetric quasi-hereditary algebra.
It would be interesting to understand the algebra 
$\mathfrak{C}(A^{\mathfrak{p}})$. Note that the natural grading filtration
on $A^{\mathfrak{p}}$ does not have to coincide with the radical 
filtration. 

In the special case $\mathfrak{g}=\mathfrak{sl}_2$ and
$\mathfrak{p}=\mathfrak{h}\oplus \mathfrak{n}_+$, the algebra 
$\mathfrak{C}(A^{\mathfrak{p}})$ is closely related to the
algebras from \cite{MT1}.
}
\end{example}

\vspace{1cm}

\noindent
Volodymyr Mazorchuk, Department of Mathematics, Uppsala University,
Box 480, 751 06, Uppsala, SWEDEN, {\tt mazor\symbol{64}math.uu.se}\\
http://www.math.uu.se/$\tilde{\hspace{1mm}}$mazor/.
\vspace{0.8cm}

\noindent
Vanessa Miemietz, Mathematical Institute, University of Oxford,
24-29 St Giles, Oxford OX1 3LB, UK, {\tt miemietz\symbol{64}maths.ox.ac.uk}\\
http://people.maths.ox.ac.uk/miemietz/.

\end{document}